\documentclass[12pt,oneside]{amsart}
\usepackage{amsmath,amssymb,latexsym,soul,cite,mathrsfs}
\usepackage{enumitem}
\usepackage{tikz}
\usepackage{lmodern}
\usepackage{ragged2e} 
\usepackage{color,enumitem,graphicx}
\usepackage[colorlinks=true,urlcolor=blue,
citecolor=red,linkcolor=blue,linktocpage,pdfpagelabels,
bookmarksnumbered,bookmarksopen]{hyperref}
\usepackage[english]{babel}
\usepackage{lineno}
\usepackage[left=2.4cm,right=2.4cm,top=2.5cm,bottom=2.4cm]{geometry}

\usepackage[hyperpageref]{backref}
\numberwithin{equation}{section}
\newtheorem{theorem}{Theorem}
\newtheorem{lemma}[theorem]{Lemma}

\newtheorem{example}{Example}

\title[On a gradient term for second-order PDEs]{On a gradient term for a class of second-order PDEs and applications to the infinity Laplace equation}

\author[J.F.\ de Oliveira]{Jos\'{e} Francisco de Oliveira}
\address[J.F.\ de Oliveira]{
\newline\indent Department of Mathematics
	\newline\indent 
	Federal University of Piau\'{i}
	\newline\indent
	64049-550 Teresina, PI, Brazil}
	\email{\href{mailto:jfoliveira@ufpi.edu.br}{jfoliveira@ufpi.edu.br}}
	

\begin{document}
\maketitle
\begin{abstract}
We propose a natural gradient term for a class of second-order partial differential equations of the form
\begin{equation}\nonumber
  M(x,Du,D^2u)+g(u)N(x,Du, D^2u)+f(x,u)=0  \;\;\mbox{in}\;\; \Omega,
\end{equation}
where  $\Omega\subset\mathbb{R}^n$ is an open set, $f\in C(\Omega\times \mathbb{R}, \mathbb{R})$,   $M$ defines  the partial differential operator, $N$ is a quadratic term driven by the gradient $Du$ and $M$ itself, and $g\in C(\mathbb{R},\mathbb{R})$. We establish conditions on the class of operators $M$ for the existence of a change of variables $v = \Phi(u)$ that transforms the previous equation into another one of the form 
\begin{equation}\nonumber  
   M(x,Dv, D^2v) + h(x,v)=0 \quad \text{in} \;\; \Omega
\end{equation}  
which does not depend on the quadratic term $N$. The results presented here unify previous findings for the Laplacian, $m$-Laplacian, and 
$k$-Hessian operators, which were derived separately by different authors and are restricted to $C^2$ solutions with fixed sign. Our work provides a more general framework, extending these findings to a broader class of nonlinear partial differential equations, including the infinity-Laplacian o\-pe\-ra\-tor. In addition, we also include both $C^2$ and viscosity solutions that may change sign. As an application, we also obtain an Aronsson-type result and investigate viscosity solutions for the Dirichlet problem associated with the infinity Laplace equation with its natural gradient term.
\end{abstract}
\section{Introduction and mains results} 
\noindent Let $\Omega$ be an open subset of $\mathbb{R}^n$, and denote by $S_n(\mathbb{R})$ the set of real symmetric $n\times n$ matrices. We investigate a natural gradient term for the class of second-order nonlinear PDEs
\begin{equation}\label{Geq1}
M(x,Du,D^2u) + g(u) N(x,Du,D^2u) + f(x,u) = 0 \;\text{in}\;\Omega,
\end{equation}
where $f:\Omega\times\mathbb{R}\to\mathbb{R}$ and $g:\mathbb{R}\to\mathbb{R}$ are continuous, 
$M,N:\Omega\times\mathbb{R}^n\times S_n(\mathbb{R})\to\mathbb{R}$, with $M$ continuous and admitting the partial derivatives $\partial^{ij}_X M(x,p,X)$ for $X=[X_{ij}]\in S_n(\mathbb{R})$, 
and $N$ strongly depending on the gradient $Du$. In fact,  $N$ is the quadratic form generated the ``gradient'' of $M$ in the last $n^2$ variables at $(x,p,X)$, i.e.   
\begin{equation}\label{NaturalG}
     N(x,p, 
X)=\langle \partial_{X}M(x,p,X)p,p\rangle,\;\;\mbox{for all}\;\; (x,p,X)\in \Omega\times\mathbb{R}^n\times S_n(\mathbb{R}),
\end{equation}
 where   $\partial_{X} M(x,p,X)\in S_n(\mathbb{R})$ is the $n\times n$ matrix whose $ij$ entry is 
\begin{equation}\label{Qmatrix}
\partial ^{ij}_{X}M(x,p,X)=\frac{\partial}{\partial X_{ij}}M(x, p, X)\;\;\mbox{if} \;\; p\in\mathbb{R}^n\setminus\{0\}\;\;\mbox{and}\;\;\partial ^{ij}_{X}M(x,0,X)=0.
\end{equation}
In view of the invariance criterion proved below, we say that $N$ is the natural gradient term for the equation \eqref{Geq1}.
For our purposes, we will assume the following hypotheses on  $M$:
\begin{enumerate}
  \item [\hypertarget{h1}{$(h_1)\;$}] $M:\Omega\times\mathbb{R}^n\times S_n(\mathbb{R})\to \mathbb{R}$ is $\alpha$-homogeneous in $p$, i.e., there exists $\alpha\in\mathbb{R}$ such that  
  $$M(x,\lambda p, X)=|\lambda|^{\alpha} M(x,p, X), \;\;\mbox{for all} \;\;\lambda\in\mathbb{R}\setminus\{0\};$$
  \item [\hypertarget{h2}{$(h_2)\;$}] For $N$ as in \eqref{NaturalG}, there exists $\beta\in\mathbb{N}\cup\{0\}$ such that
\[
M(x, p, \gamma X + \sigma (p\otimes p)) = \gamma^{\beta+1} M(x, p, X) + \sigma\gamma^{\beta}N(x, p, X),
\]
for any scalars $\gamma \in \mathbb{R}\setminus\{0\}$ and $\sigma \in \mathbb{R}$,   where $p\otimes q$ is the tensor product $q^{T}p$, or  $[q_ip_j]$ in the matrix form, of the vectors $p,q\in\mathbb{R}^n$.
\end{enumerate}
We note that  \hyperlink{h1}{$(h_1)$} and \hyperlink{h2}{$(h_2)$}  are satisfied for a wide class of operators, such as the Laplacian, the $m$-Laplacian, the $k$-Hessian, the $1$-Laplacian and the infinity-Laplacian. So, the general class $M$ includes both divergent and non-divergent form operators (cf. Section~\ref{sec-ex} for details).

On the function $g\in C(\mathbb{R},\mathbb{R})$ we shall assume the following assumption.
\begin{enumerate}
  \item [\hypertarget{g0}{$(g_0)\;$}] there exists $s_0\ge0$ such that $g\ge 0$ in $[s_0,\infty)$ and $g\le 0$ on $(-\infty, -s_0]$.
\end{enumerate}
Taking into account hypothesis \hyperlink{g0}{$(g_0)$}, we  define the change of variable $\Phi_{g}:\mathbb{R}\to\mathbb{R}$  given by
 \begin{equation}\label{KKC}
     \Phi_{g}(t)=\int_{0}^{t}\operatorname{e}^{G(s)}ds, \quad \mbox{where}\quad G(s)=\int_{0}^{s}g(\tau)d\tau.
 \end{equation}
Note that,  for $t\le 0$, we have $ \Phi_{g}(t)=-\int_{t}^{0}\operatorname{e}^{G(s)}ds$ and $G(s)=-\int_{s}^{0}g(\tau)d\tau$. In addition, condition \hyperlink{g0}{$(g_0)$} ensures that $\Phi_g$ has a characteristic profile. Specifically, $\Phi_g$ is a strictly increasing $C^2$-diffeomorphism on $\mathbb{R}$ that is concave on $(-\infty, s_0]$, convex on $[s_0, \infty)$, and satisfies $\Phi_{g}(-\infty)=-\infty$, $\Phi_{g}(0)=0$,  and $\Phi_{g}(+\infty)=+\infty$, see Lemma~\ref{C2-change} below.

\subsection{$C^2$-solutions case}
\label{sec2}
We established the invariance of the natural gradient term $N$ for  partial differential equations of the form $M+gN+f=0$ that works for $C^2$-solutions.
\begin{theorem}\label{theorem1} Let $f\in C(\Omega\times \mathbb{R}, \mathbb{R})$ and let  $\Phi_{g}$ be given by \eqref{KKC} with $g\in C(\mathbb{R},\mathbb{R})$ such that \hyperlink{g0}{$(g_0)$} holds.  Suppose that  $M$ satisfies \hyperlink{h1}{$(h_1)$} and \hyperlink{h2}{$(h_2)$}. If $u\in C^{2}(\Omega)$ is  solution  for the equation
\begin{equation}\label{ECG}
    M(x,Du,D^2u)+g(u)N(x,Du, D^2u)+f(x,u)=0
\end{equation} then $v=\Phi_g(u)$ is also a solution to
\begin{equation}\label{EWG}
    M(x,Dv, D^2v)+h(x,v)=0,\;\;\mbox{with} \;\; h(x,s)=\mathrm{e}^{(\alpha+\beta+1)G(\Phi^{-1}_{g}(s))}f(x, \Phi^{-1}_{g}(s)).
\end{equation}
Reciprocally, if $v\in C^{2}(\Omega)$ solves the equation \eqref{EWG}, then $u=\Phi^{-1}_g(v)$ solves \eqref{ECG}.
\end{theorem}
Theorem~\ref{theorem1} is mainly motivated by the works by \cite{MET2,Arcoya} for the Laplace equation, the recent papers \cite{ubilla} for the  $m$-Laplace equation, \cite{Monge} for the Monge-Ampère equation, and more generally \cite{Nosso} for the $k$-Hessian equation. All of them have treated the corresponding equation in the presence of a natural gradient-type term and have explored a Kazdan-Kramer \cite{kazdan} type change of variables to transform the original equation into another without the gradient term. Theorem~\ref{theorem1} unifies the approach for these operators and extends it to a general operator defined by a function $M$. It is worth noting that our choice for the gradient term  $N$ is intrinsically related to the operator through hypothesis \hyperlink{h2}{$(h_2)$}.

We observe that the properties $\Phi_{g}(-\infty)=-\infty$, $\Phi_{g}(0)=0$,  and $\Phi_{g}(+\infty)=+\infty$ ensure that the either homogeneous $u_{|_{\partial\Omega}}=0$ or blow-up $u_{|_{\partial\Omega}}=\pm\infty$ boundary conditions for \eqref{ECG}-\eqref{EWG} are preserved by the Theorem~\ref{theorem1}. In addition, if we assume that the solution $u \in C^{2}(\Omega)$ is positive (or negative),  the result still holds under the weaker condition $g\in C([0,\infty),\mathbb{R}_{+})$ (or $g\in C((-\infty,0]),\mathbb{R}_{-}))$ instead of \hyperlink{g0}{$(g_0)$}. We recall that the $k$-admissible solutions the $k$-Hessian boundary value problem are often negative, see \cite{Tso,Nosso} for details, while for the Laplace and $m$-Laplace, $2\le m<\infty$, the $C^2$ solutions are often positive, see \cite{MET2, Arcoya, ubilla}.

\subsection{Viscosity solutions case}
The hypotheses \hyperlink{h1}{$(h_1)$} and \hyperlink{h2}{$(h_2)$} reach a broad class of second-order nonlinear equations, including  degenerate elliptic operators such as the infinity Laplacian. In this scenario,  we do not expect the existence of $C^2$-solutions, as observed by several authors \cite{Yu,Savin,Evans,EvansSavin, EvansSmart} since the classic paper by Aronsson \cite{Aronsson}. To address this challenge, we present below an extension of Theorem~\ref{theorem1} that is suitable for the viscosity solutions theory. Firstly, we will say a few words about the notion of viscosity solutions, for a deeper discussion we recommend \cite{Cabre,Ishii, GuideV}.

Let  $\mathcal{F}: \Omega \times\mathbb{R} \times \mathbb{R}^n \times S_{n}(\mathbb{R}) \to \mathbb{R}$ be a general second-order operator. We say that continuous function $u:\Omega\to\mathbb{R}$ is a viscosity subsolution of the equation
\begin{equation}\label{V-S}
    \mathcal{F}(x, w, Dw, D^2w)= 0 \quad \text{in } \Omega,
\end{equation}
if whenever $\varphi \in C^2(\Omega)$ is such that $ u - \varphi$ has a local maximum at some point $x_0 \in \Omega$, then there holds
\[
    \mathcal{F}(x_0, u(x_0), D\varphi(x_0), D^2 \varphi(x_0)) \le 0.
\]
Similarly, a continuous function \( u: \Omega \to \mathbb{R} \) is called a viscosity supersolution of equation \eqref{V-S}, if $\varphi \in C^2(\Omega)$ is such that $u-\varphi$ has a local minimum at some point $x_0 \in \Omega $, then there holds
$$
     \mathcal{F}(x_0, u(x_0), D\varphi(x_0), D^2 \varphi(x_0) \ge 0.
$$
We say that $ u$ is a viscosity solution of the equation \eqref{V-S} when $ u$ is both a subsolution and a supersolution.
\begin{theorem}\label{thm-VS}  Suppose that $f$, $g$, $h$,  $\Phi_g$,  and $M$ are as in Theorem~\ref{theorem1}. If $v:\Omega\to\mathbb{R}$ is a viscosity solution of the equation
\begin{equation}\label{EWGV}
    M(x,Dw,D^2w)+h(x,w)=0 \;\;\mbox{in}\;\;\Omega
\end{equation}
then $u=\Phi^{-1}_g(v)$ is a viscosity solution of
\begin{equation}\label{ECGV}
  M(x,Dw,D^2w)+g(w)N(x,Dw, D^2w)+f(x,w)=0 \;\;\mbox{in}\;\;\Omega.
\end{equation} 
Conversely, if $u\in C(\Omega)$  is a viscosity solution of the equation \eqref{ECGV}, then $v=\Phi_g(u)$ is a viscosity solution \eqref{EWGV}.
\end{theorem}
As far as we know, Theorem~\ref{thm-VS} represents the first advance toward determining invariant gradient term that work for viscosity solutions. This substantially broadens the approach to these types of problems, allowing for the investigation of operators for which the existence of $C^2$-solutions is not expected.

This paper is organized as follows. In Section~\ref{sec-ex}, we present some examples of operators covered by our results. In Section~\ref{sec3}, we present applications of Theorem~\ref{theorem1} and Theorem~\ref{thm-VS} to the Dirichlet problem for the infinity Laplacian operator. Finally, in Section~\ref{sec-proof}, we provide the proofs of Theorem~\ref{theorem1} and Theorem~\ref{thm-VS}.
\section{Examples of operators covered by our results}\label{sec-ex}
We present below some PDEs which can be handled by Theorem~\ref{theorem1} and Theorem~\ref{thm-VS}.
\subsection{Laplace operator}\label{subsec-Laplace} For the Laplace operator $\Delta u=\mathrm{tr}(D^2u)$, where $\mathrm{tr}$ denotes the trace function on $S_n(\mathbb{R})$,  we have $M(x,p, X)=\mathrm{tr}(X)$
for $(x,p, X)\in \Omega\times\mathbb{R}^n\times S_n(\mathbb{R})$. Note that $N(x,p,X)=|p|^2$ and thus  $M$ satisfies  \hyperlink{h1}{$(h_1)$} and \hyperlink{h2}{$(h_2)$} with  $\alpha=0$ and $\beta=0$. The corresponding equation $M+gN+f=0$ in \eqref{Geq1} becomes
\begin{equation}\label{LaplaceDelta}
   \Delta u+ g(u)|Du|^{2}+f(x,u)=0. 
\end{equation}
For an in-depth discussion on the equation \eqref{LaplaceDelta}  under different conditions on $g$ and $f$, we recommend \cite{Arcoya,MET2,Garcia,JEANJEAN1,JEANJEAN2} and the references therein.
\subsection{$m$-Laplace operator, $m\in [1,\infty)$}\label{subsec-mLaplace} For the $m$-Laplace operator $\Delta_mu=\mathrm{div}(|Du|^{m-2}Du)$ we have  $M(x,0,X)=0$ and, if $(x,p,X)\in \Omega\times(\mathbb{R}^n\setminus\{0\})\times S_n(\mathbb{R})$ 
$$M(x,p, X)=|p|^{m-4}\Big(|p|^{2}\operatorname{tr}{X}+(m-2)\langle Xp, p\rangle\Big).$$
Hence
$$
\partial_{X}M(x,p,X) = |p|^{m-4}\Big(|p|^{2}I+(m-2)(p \otimes p)\Big)\;\; \mbox{and}\;\; N(x,p, X)= (m-1)|p|^{m}.
$$
Thus,  $M$ satisfies  \hyperlink{h1}{$(h_1)$} and \hyperlink{h2}{$(h_2)$} with  $\alpha=m-2$ and $\beta=0$. Here, the 
equation \eqref{Geq1} becomes   
\begin{equation}\label{E-Plaplace}
  \Delta_m u+ g(u)(m-1)|Du|^{m}+f(x,u)=0.   
\end{equation}
Note that, according with our approach,  the $1$-Laplace operator does not admit invariant non-trivial gradient term. The $m$-Laplace equation with gradient term \eqref{E-Plaplace} was recently investigated in \cite{ubilla,Adimurthi1} and \cite{DeCoster}.

\subsection{$k$-Hessian operators, $1\le k\le n$)}\label{subsec-kHessian} For the $k$-Hessian operator we have  $M(x,p, X)=\mathrm{tr}_{k}(X)$ is the $k$-trace of $X$, that is, the sum of all $k\times k$ principal minors of the matrix $X$. From the properties of the $k$-trace in \cite{IvoI, IvoII}, for any $X\in S_n(\mathbb{R})$ and $p\in\mathbb{R}^n$ we have
\begin{equation}\label{k-tracep}
 \mathrm{tr}_k(X\pm p\otimes p)=\mathrm{tr}_k(X)\pm \langle S^{*}_k(X)p,p \rangle.    
\end{equation}
where $S^{*}_k(X)=[S^{ij}_k(X)]$  with $S^{ij}_k(X)=\frac{\partial}{\partial X_{ij}}\mathrm{tr}_k(X)$. Noticing that $\mathrm{tr}_k(\gamma X)=\gamma^k\mathrm{tr}_k(X)$ the chain rule yields $S^{*}_k(\gamma X)=\gamma^{k-1}S^{*}_k(X)$, for $\gamma\in\mathbb{R}\setminus\{0\}$.  So, we obtain
\begin{equation*}
   M(x,p,\gamma X +\sigma(p\otimes p))=\gamma^{k} M(x,p,X)+\gamma^{k-1}\sigma\langle \partial_{X}N(x,p,X)p,p \rangle.
\end{equation*}
It follows that $M$ satisfies  \hyperlink{h1}{$(h_1)$} and \hyperlink{h2}{$(h_2)$} with  $\alpha=0$ and $\beta=k-1$. The $k$-Hessian equation with gradient-type term was recently studied in \cite{Monge} for $k=n$ and \cite{Nosso} for general $1\le k\le n$.\\

\subsection{Infinity-Laplacian operator}\label{subsec-infinityLaplace} For the infinity-Laplacian  operator $\Delta_{\infty}u=\langle D^2u Du, Du\rangle$,  we have $M(x, p, X) = \langle Xp, p\rangle
$, for all $(x,p, X)\in \Omega\times\mathbb{R}^n\times S_n(\mathbb{R})$. Then,  
\begin{equation}\label{infityGL}
    \partial_{X}M(x,p,X)=p\otimes p\;\;\mbox{and}\;\; N(x,p,X)=|p|^4.
\end{equation}
The assumptions \hyperlink{h1}{$(h_1)$} and \hyperlink{h2}{$(h_2)$} are satisfied with $\alpha=2$ and $\beta=0$. In this case,  Theorem~\ref{theorem1} can be used to investigate the equation
\begin{equation}\label{InfinityL-turbinado}
\Delta_{\infty}u+g(u)|Du|^4+f(x,u)=0
\end{equation}
through the study of the simpler equation
\begin{equation}\label{InfinityL-puro}
  \Delta_{\infty}v+h(x,v)=0.  
\end{equation}
The infinity-Laplacian operator serves as a canonical singular operator with diverse applications. Its theory is derived from absolutely minimizing Lipschitz extensions and $L^{\infty}$ variational problems, connecting it to optimal transport \cite{Aronsson2} -- a framework that provides the mathematical basis for applications in digital image processing \cite{Caselles}. Moreover, its characterization as the continuous limit of certain stochastic tug-of-war games establishes a connection to the fields of game theory and control theory \cite{Peres,Figa}.
  
The equation \eqref{InfinityL-puro} has been investigated by several authors in different contexts, see \cite{Aronsson,GuideV,Yu,Bhatta,Bhatta2,Damiao,GLu} and the references therein. However, \eqref{InfinityL-turbinado} remains unaddressed in the literature. In this work, we provide existence, uniqueness, and non-existence results  for the  infinity Laplacian equation \eqref{InfinityL-turbinado} in the presence of a gradient term (cf. Section~\ref{sec3}).
\subsection{Game-theoretic operator}\label{subsec-game} For the game theoretic or inhomogeneous normalized infinity-Laplacian $\Delta^{\mathcal{N}}_{\infty}u=\frac{1}{|Du|^2}\Delta_{\infty}u$, we can take $M(x, p, X) = \frac{1}{|p|^2}\langle Xp, p\rangle
$, for all $(x,p, X)\in \Omega\times(\mathbb{R}^n\setminus\{0\})\times S_n(\mathbb{R})$ and $M(x,0,X)=0$. So,    
$\partial_{X}M(x,p,X)=\frac{1}{|p|^2}(p\otimes p)\;\;\mbox{and}\;\; N(x,p,X)=|p|^2$ and $M$ satisfies \hyperlink{h1}{$(h_1)$} and \hyperlink{h2}{$(h_2)$} with $\alpha=0$ and $\beta=0$. Thus, Theorem~\ref{theorem1} allows us to treat the  equation
\begin{equation}\label{game-op}
\Delta^{\mathcal{N}}_{\infty}u+g(u)|Du|^2+f(x,u)=0.
\end{equation}
For a comprehensive treatment of the game-theoretic operator, we refer the reader to \cite{Peres,Figa} and the references therein.
\section{Applications to infinity Laplace equation}\label{sec3}
In this section, we will apply Theorem~\ref{theorem1} and Theorem~\ref{thm-VS} to investigate new classes of equations containing a gradient term for the infinity-Laplacian operator.  In what follows, we assume that $g\in C(\mathbb{R}, \mathbb{R})$ such that \hyperlink{g0}{$(g_0)$} holds.

Firstly, we apply Theorem~\ref{theorem1} to obtain an extension of Aronsson result, see \cite{Aronsson,Yu}.
\begin{theorem}\label{app-Aron}
  Let $\Omega$ be  a connected open subset of $\mathbb{R}^n$.  Suppose that $u\in C^{2}(\Omega)$ solves
$$
\Delta_{\infty}u+g(u)|Du|^4=0.
$$  
Then, either $Du\not=0$ in $\Omega$ or $u$ reduces to a constant.  
\end{theorem} 
\begin{proof}
Let $u\in C^2(\Omega)$ be a solution of $\Delta_{\infty}u+g(u)|Du|^4=0$. Then, taking into account \eqref{infityGL} in the Section~\ref{subsec-infinityLaplace},  Theorem~\ref{theorem1} tell us that $v=\Phi_g(u)$ is a $C^{2}(\Omega)$ solution of the $\Delta_{\infty}v=0$. By the Aronsson theorem \cite{Aronsson,Yu}, we obtain that either $Dv\not=0$ in $\Omega$ or $v$ reduces to a constant. Since $Dv=D(\Phi_{g}(u))=\operatorname{e}^{G(\Phi_{g}(u))}Du$, we also have that either $Du\not=0$ in $\Omega$ or $u$ reduces to a constant.
\end{proof}
Next, we explore the discussion in \cite{Bhatta,Bhatta2} to provide some conditions for the existence, non-existence and uniqueness of viscosity solutions for the inhomogeneous Dirichlet problem for infinity Laplacian equation with gradient term
\begin{equation}\label{Dirichle P}
    \left\{
    \begin{aligned}
     \Delta_{\infty}u+g(u)|Du|^4 +f(x,u)&=0\;&\mbox{in}&\;\; \Omega\\
    u&=b\;&\mbox{on}&\;\; \partial\Omega
    \end{aligned}
    \right.
\end{equation}
where $\Omega\subset\mathbb{R}^n$, $n\ge 2$ is a bounded domain,  $b\in C(\partial\Omega)$ and $f\in C(\Omega\times\mathbb{R}, \mathbb{R})$. In addition, we will consider the following conditions on the pair $f$ and $g$.
\begin{enumerate}
  \item [\hypertarget{f0}{$(f_0)$}] for every compact interval $I\subset\mathbb{R}$
  $$\sup_{\Omega\times I}|f(x,t)|<\infty.$$
  \item [\hypertarget{fg}{$(fg)$}] for some $\nu,\xi\in [-\infty,0]$ we have
  \[
  \limsup_{t \to \infty} \frac{\sup_{\Omega} \operatorname{e}^{3G(t)}f(x,t)}{\Big(\int_{0}^{t}\operatorname{e}^{G(s)}ds\Big)^3} = \nu \quad\mbox{and}\quad \limsup_{t \to -\infty} \frac{\inf_{\Omega} \operatorname{e}^{3G(t)}f(x,t)}{\Big(\int_{0}^{t}\operatorname{e}^{G(s)}ds\Big)^3} = \xi
  \]
  where $G(s)=\int_{0}^{s}g(\tau)d\tau$.
\end{enumerate}
\begin{theorem}\label{app-exist} Suppose \hyperlink{f0}{$(f_0)$} and \hyperlink{fg}{$(fg)$} hold.  Then  the Dirichlet problem \eqref{Dirichle P} admits a viscosity solution $u\in C(\overline{\Omega})$.  
\end{theorem}
\begin{proof}
    Let us define $h(x,s)=\mathrm{e}^{3G(\Phi^{-1}_{g}(s))}f(x, \Phi^{-1}_{g}(s))$, where $\Phi^{-1}_{g}$ is the inverse of the change $\Phi_{g}$ in \eqref{KKC}. By continuity of $\Phi^{-1}_{g}$  and from \hyperlink{f0}{$(f_0)$} we have that the function $h_0=-h$ satisfies
    \begin{equation}\label{h0-bound}
        \sup_{\Omega\times I}|h_0(x,s)|=\sup_{\Omega\times I}\mathrm{e}^{3G(\Phi^{-1}_{g}(s))}|f(x, \Phi^{-1}_{g}(s))|<\infty,
    \end{equation}
    for any compact interval $I\subset\mathbb{R}$. In addition, by setting the change $t=\Phi^{-1}_{g}(s)$ and using \hyperlink{fg}{$(fg)$} we have that
    \begin{align*}
        \limsup_{s \to \infty} \frac{\sup_{\Omega} h(x,s)}{s^3}&= \limsup_{s \to \infty} \frac{\sup_{\Omega} \mathrm{e}^{3G(\Phi^{-1}_{g}(s))}f(x, \Phi^{-1}_{g}(s))}{s^3}\\
        &=\limsup_{t \to \infty} \frac{\sup_{\Omega} \mathrm{e}^{3G(t)}f(x, t)}{\Big(\int_{0}^{t}\operatorname{e}^{G(s)}ds\Big)^3}=\nu.
    \end{align*}
    Analogously,
     \begin{align*}
        \limsup_{s \to -\infty} \frac{\inf_{\Omega} h(x,s)}{s^3}=\xi.
    \end{align*}
    It follows that 
    \begin{align}\label{h01}
        \liminf_{s \to \infty} \frac{\inf_{\Omega} h_0(x,s)}{s^3}&= \liminf_{s \to \infty} \frac{\inf_{\Omega} -h(x,s)}{s^3}=- \limsup_{s \to \infty} \frac{\sup_{\Omega}h(x,s)}{s^3}=-\nu
    \end{align}
    and
     \begin{align}\label{h02}
        \liminf_{s \to -\infty} \frac{\sup_{\Omega} h_0(x,s)}{s^3}&= \liminf_{s \to -\infty} \frac{\sup_{\Omega} -h(x,s)}{s^3}=- \limsup_{s \to -\infty} \frac{\inf_{\Omega}h(x,s)}{s^3}=-\xi.
    \end{align}
    Hence, by combining \eqref{h0-bound}, \eqref{h01} and \eqref{h02}, we get that \cite[Theorem~5.5]{Bhatta} ensures the existence of a viscosity solution $v\in C(\overline{\Omega})$ for the Dirichlet problem
    \begin{equation}\label{DP-simpes}
    \left\{
    \begin{aligned}
     \Delta_{\infty}v&=h_0(x,v)\;&\mbox{in}&\;\; \Omega\\
    v&=\tilde{b}\;&\mbox{on}&\;\; \partial\Omega
    \end{aligned}
    \right.
\end{equation}
where $\tilde{b}\in C(\partial\Omega)$ is given by $\tilde{b}=\Phi_g(b)$. Finally, from Theorem~\ref{thm-VS} and Section~\ref{subsec-infinityLaplace}, we conclude that $u=\Phi^{-1}_g(v)$ is a viscosity solution of \eqref{Dirichle P} as desired.
\end{proof}
The next result guarantees the uniqueness of solution for \eqref{Dirichle P} when $b=c$ is a constant  and $f(x,u)=\bar{f}(u)$ is independent of variable $x$.
\begin{theorem}
  Assume that $f(x,t)=\bar{f}(t)$ such that $t\mapsto \bar{f}(t)\operatorname{e^{3G(t)}}$ is a non-increasing function and $b=c$ is  constant. Then the Dirichlet problem \eqref{Dirichle P} has at most one solution.
\end{theorem}
\begin{proof}
Let $h(s)=\mathrm{e}^{3G(\Phi^{-1}_{g}(s))}\bar{f}(\Phi^{-1}_{g}(s))$, where $\Phi^{-1}_{g}$ is the inverse of $\Phi_{g}$ in \eqref{KKC}. Since  $\Phi^{-1}_{g}$ is increasing and by assumption $t\mapsto \bar{f}(t)\operatorname{e^{3G(t)}}$ is non-increasing, we have that $h$ is also a non-increasing function. So, $h_0=-h$ is non-decreasing. According with \cite[Corollary 7.5]{Bhatta} the Dirichlet problem
\begin{equation}\label{DP-Unic}
    \left\{
    \begin{aligned}
     \Delta_{\infty}v&=h_0(v)\;&\mbox{in}&\;\; \Omega\\
    v&=\tilde{c}\;&\mbox{on}&\;\; \partial\Omega,
    \end{aligned}
    \right.
\end{equation}
where  $\tilde{c}$ is any constant admits at most one solution. On the other hand, from Theorem~\ref{thm-VS}, for each solution $u$ of \eqref{Dirichle P}, the function $v=\Phi_g(u)$ is a solution of \eqref{DP-Unic} with $\tilde{c}=\Phi_g(c)$ . Since $\Phi_g$ is strictly increasing, the uniqueness of the solution for \eqref{Dirichle P} follows from the problem \eqref{DP-Unic}.
\end{proof}
To discuss our non-existence result, we assume that $f$ is non-negative, that is, $f:\Omega \times \mathbb{R}\to  [0, \infty)$ is  continuous and satisfies  \hyperlink{f0}{$(f_0)$}. In addition,  for $\ell=\inf_{\partial\Omega}\Phi_g(b)$ we define
\begin{align}\label{etafg}
    \eta_{f,g}(t)=\inf_{\Omega\times[t,\infty)}\operatorname{e}^{3G(\Phi^{-1}_{g}(s))}f(x,\Phi^{-1}_{g}(s)),\;\;\mbox{for every}\;\; t\ge \ell
\end{align}
where $\Phi^{-1}_{g}$ is the inverse of $\Phi_{g}$ in \eqref{KKC}.
Without loss of generality, we can assume in our subsequent discussion  that $\eta_{f,g}$ is continuous in $[\ell, \infty)$, see \cite{Bhatta} for details. We adopt the following two assumptions
\begin{enumerate}
  \item [\hypertarget{eta0}{$(\eta_0)$}] $\eta_{f,g}(t)>0$ for all $t>\ell$
  
  \item [\hypertarget{eta1}{$(\eta_1)$}] $\displaystyle\sup_{a>\ell}\zeta_{f,g}(a)=S_{f,g}<\infty$, where
  \[
  \zeta_{f,g}(a)=\int_{\ell}^{a}\frac{1}{\sqrt[4]{H_{f,g}(a)-H_{f,g}(t)}}dt,\;\;\mbox{for}\;\; a>\ell
  \]
  with $H_{f,g}(t)=\int_{\ell}^{t}\eta_{f,g}(s)ds$  for $t>\ell$.
\end{enumerate}
\begin{theorem}\label{thm-nonexistence}
   Let $B_R(z)$ be the in-ball of $\Omega$. Assume that $f$ and $g$ are such that $\eta_{f,g}$ in \eqref{etafg}  satisfies  \hyperlink{eta0}{$(\eta_0)$} and \hyperlink{eta1}{$(\eta_1)$}. Then the Dirichlet problem \eqref{Dirichle P} has no solutions $u\in C(\overline{\Omega})$, besides, possibly the constant solution, if $R>S_{f,g}/\sqrt{2}$.
\end{theorem}
\begin{proof}
    By contradiction, suppose that $u\in C(\overline{\Omega})$ solves the problem \eqref{Dirichle P}. Thus, Theorem~\ref{thm-VS} ensures that $v=\Phi_g(u)$ is a solution of  the  problem 
    \begin{equation}\label{DP-non}
    \left\{
    \begin{aligned}
     \Delta_{\infty}v+h&=0\;&\mbox{in}&\;\; \Omega\\
    v&=\tilde{b}\;&\mbox{on}&\;\; \partial\Omega
    \end{aligned}
    \right.
\end{equation}
where  $\tilde{b}=\Phi_g(b)\in C(\partial\Omega)$ and $h(x,s)=\operatorname{e}^{3G(\Phi^{-1}_{g}(s))}f(x,\Phi^{-1}_{g}(s))$. From \hyperlink{eta0}{$(\eta_0)$} and \hyperlink{eta1}{$(\eta_1)$} we obtain  that the function $h\in C(\Omega\times\mathbb{R}, [0,\infty))$ satisfies all the assumption of \cite[Theorem~4.1]{Bhatta}. Consequently, \eqref{DP-non} does not admit solution for $R>S_{f,g}/\sqrt{2}$. This contradiction completes the proof.  
\end{proof}
\begin{example} As an example of a pair of functions $f$ and $g$ that satisfies \hyperlink{g0}{$(g_0)$}, \hyperlink{f0}{$(f_0)$}, \hyperlink{eta0} {$(\eta_0)$}, and \hyperlink{eta1} {$(\eta_1)$} we can take $f:\Omega\times\mathbb{R}\to \mathbb{R}$ and $g:\mathbb{R}\to\mathbb{R}$ given by
\[
f(x,t)=\frac{{\operatorname{e}^{t+\frac{t^{3}}{3}}}}{(1+t^2)^3} \quad \mbox{and}\quad g(t)=\frac{2t}{(1+t^2)}.
\]
In fact,  we have that $\eta_{f,g}$ in \eqref{etafg} becomes $\eta_{f,g}(t)=e^{t}$. Of course, \hyperlink{eta0} {$(\eta_0)$} holds. Moreover, from \cite[Lemma 8.1]{Bhatta} we can see that  \hyperlink{eta1} {$(\eta_1)$}  holds.
    
\end{example}
\section{Proof of Theorem~\ref{theorem1} and Theorem~\ref{thm-VS}} \label{sec-proof}
In this section, we prove  Theorem~\ref{theorem1} and Theorem~\ref{thm-VS}. We start by employing the assumptions \hyperlink{h1}{$(h_1)$} and \hyperlink{h2}{$(h_2)$} to get the following key result. 
\begin{lemma}\label{A-change} Suppose that the operator $M$ satisfies \hyperlink{h1}{$(h_1)$} and \hyperlink{h2}{$(h_2)$}.  Let $u\in C^2(\Omega)$ and let $\Phi: I\to\mathbb{R}$ be a strictly increasing $C^2$-function defined on an interval $I\supset\{u(x)\;:\; x\in \Omega\}$. Then
\begin{equation}
    M(x, D\Phi(u), D^2\Phi(u))= [\Phi^{\prime}(u)]^{\alpha+\beta+1}M(x,Du, D^2u)+[\Phi^{\prime}(u)]^{\alpha+\beta}\Phi^{\prime\prime}(u)N(x, Du, D^2u).
\end{equation}  
\end{lemma}
\begin{proof}
It is easy to verify that     
$$D\Phi(u)=\Phi^{\prime}(u)Du \;\;\mbox{and}\;\; D^2\Phi(u)=\Phi^{\prime}(u)D^2u+\Phi^{\prime\prime}(u)Du\otimes Du.$$
Hence, from \hyperlink{h1}{$(h_1)$} and \hyperlink{h2}{$(h_2)$} we have
\begin{align*}
  M(x, D\Phi(u), D^2\Phi(u))&=M(x,\Phi^{\prime}(u)Du, \Phi^{\prime}(u)D^2u+\Phi^{\prime\prime}(u)Du\otimes Du)\\ 
  &=[\Phi^{\prime}(u)]^{\alpha}M(x,Du, \Phi^{\prime}(u)D^2u+\Phi^{\prime\prime}(u)Du\otimes Du)\\
  &=[\Phi^{\prime}(u)]^{\alpha}\Big[[\Phi^{\prime}(u)]^{\beta+1}M(x,Du, D^2u)+[\Phi^{\prime}(u)]^{\beta}\Phi^{\prime\prime}(u)N(x, Du, D^2u)\Big]
\end{align*}
as desired.
\end{proof}
\begin{lemma}\label{C2-change} Let $\Phi_g:\mathbb{R}\to\mathbb{R}$ be given by \eqref{KKC} with $g\in C(\mathbb{R},\mathbb{R})$ such that \hyperlink{g0}{$(g_0)$} holds. Then $\Phi_g$ is a strictly increasing $C^2$-diffeomorphism  such that  $\Phi_{g}(-\infty)=-\infty$, $\Phi_{g}(0)=0$,  and $\Phi_{g}(+\infty)=+\infty$. 
\end{lemma}
\begin{proof} It is clear that $\Phi_{g}(0)=0$. In addition, for  $t\in\mathbb{R}$, we have 
\begin{equation}\label{deddPhi}
  \Phi^{\prime}_{g}(t)=\operatorname{e}^{G(t)}>0\;\; \mbox{and}\;\; \Phi^{\prime\prime}_{g}(t)=g(t)\operatorname{e}^{G(t)}.  
\end{equation}
Hence,  $\Phi_g$ is a strictly increasing $C^2$ function on $\mathbb{R}$ that is concave on $(-\infty, s_0]$ and convex on $[s_0, \infty)$. In order to prove that $\Phi_g$ is  onto,  we employ assumption  \hyperlink{g0}{$(g_0)$}. Indeed, for $s\le -s_0\le 0$, we have
$$G(s)=-\int_{s}^{0}g(\tau)d\tau =\int_{s}^{-s_0}|g(\tau)|d\tau+G(-s_0)\ge G(-s_0).$$
Hence, for $t\le -s_0$, we obtain 
\begin{equation}\nonumber
    \Phi_{g}(t)=-\int_{t}^{-s_0}\operatorname{e}^{G(s)}ds-\int_{-s_0}^{0}\operatorname{e}^{G(s)}ds\le -\int_{t}^{-s_0}\operatorname{e}^{G(s)}ds\le -\operatorname{e}^{G(-s_0)}\int_{t}^{-s_0}ds.
\end{equation}
It follows that $\Phi_{g}(-\infty)=\lim_{t\to-\infty} \Phi_{g}(t)=-\infty$. Analogously, for $s\ge s_0\ge 0$, the assumption \hyperlink{g0}{$(g_0)$} yields 
$$G(s)=\int_{0}^{s}g(\tau)d\tau =\int_{0}^{s_0}g(\tau)d\tau+\int_{s_0}^{s}|g(\tau)|d\tau\ge G(s_0)$$
and thus, for $t\ge s_0$
\begin{equation}\nonumber
    \Phi_{g}(t)=\int_{0}^{s_0}\operatorname{e}^{G(s)}ds+\int_{s_0}^{t}\operatorname{e}^{G(s)}ds\ge \int_{s_0}^{t}\operatorname{e}^{G(s)}ds\ge \operatorname{e}^{G(s_0)}\int_{s_0}^{t}ds.
\end{equation}
Consequently, $\Phi_{g}(+\infty)=\lim_{t\to+\infty} \Phi_{g}(t)=+\infty$. Finally, the chain rule and \eqref{deddPhi} imply that
\begin{equation}\label{deddPhi-}
    (\Phi^{-1}_{g})^{\prime}(s)=\frac{1}{\Phi^{\prime}_g(t)} \;\; \mbox{and}\;\; (\Phi^{-1}_{g})^{\prime\prime}(s)=-\frac{\Phi^{\prime\prime}_{g}(t)}{[\Phi^{\prime}_g(t)]^{3}} 
\end{equation}
with $t=\Phi^{-1}_{g}(s)$. Hence, $\Phi^{-1}_{g}$ belongs to $C^2$, which completes the proof. 
\end{proof}
\begin{proof}[Proof of Theorem~\ref{theorem1}] From Lemma~\ref{C2-change},  $\Phi_g$ is a strictly increasing $C^2$-diffeomorphism on $\mathbb{R}$ with 
$\Phi^{\prime}_{g}(t)=\mathrm{e}^{G(t)}$, $\Phi_g(-\infty)=-\infty$, $\Phi_g(0)=0$, and $\Phi_g(+\infty)=+\infty$.  Hence, for any $u\in C^{2}(\Omega)$ solution of \eqref{ECG}, by setting $v=\Phi_g(u)$  the Lemma~\ref{A-change} yields 
\begin{equation}\nonumber
\begin{aligned}
    M(x, Dv, D^2v)&=[\Phi_{g}^{\prime}(u)]^{\alpha+\beta+1}M(x,Du, D^2u)+[\Phi_{g}^{\prime}(u)]^{\alpha+\beta}\Phi_{g}^{\prime\prime}(u)N(x, Du, D^2u)\\
    &=[\Phi_{g}^{\prime}(u)]^{\alpha+\beta}\Phi_{g}^{\prime\prime}(u)N(x, Du, D^2u)\\
    &-[\Phi_{g}^{\prime}(u)]^{\alpha+\beta+1}[g(u)N(x,Du, D^2u)+f(x,u)]\\
    &=\Big[[\Phi_{g}^{\prime}(u)]^{\alpha+\beta}\Phi_{g}^{\prime\prime}(u)-g(u)[\Phi_{g}^{\prime}(u)]^{\alpha+\beta+1}\Big]N(x, Du, D^2u)- \mathrm{e}^{(\alpha+\beta+1)G(u)}f(x,u).
\end{aligned}
\end{equation} 
Since that $u=\Phi^{-1}_g(v)$ and $t\mapsto \Phi_g(t)$ solves the ODE $$y^{\prime\prime}(t)-g(y(t))y^{\prime}(t)=0,\;\;t\in\mathbb{R}$$ we obtain
$$
 M(x, Dv, D^2v)+\mathrm{e}^{(\alpha+\beta+1)G(\Phi^{-1}_g(v))}f(x,\Phi^{-1}_g(v))=0.
$$
Hence, $v=\Phi_g(u)$ is  solution for \eqref{EWG}. Reciprocally, if $v\in C^2(\Omega)$ is a solution for \eqref{EWG}, then $u=\Phi^{-1}_g(v)$ belongs to $C^{2}(\Omega)$. In addition, 
\begin{equation}
    M(x,D\Phi_g(u), D^2\Phi_g(u))= M(x,Dv, D^2v)=-\mathrm{e}^{(\alpha+\beta+1)G(u)}f(x,u)
\end{equation}
and, from Lemma~\ref{A-change}
\begin{equation}\nonumber
    [\Phi^{\prime}_{g}(u)]^{\alpha+\beta+1}M(x,Du, D^2u)+[\Phi^{\prime}_{g}(u)]^{\alpha+\beta}\Phi^{\prime\prime}_{g}(u)N(x,Du, D^2u)+\mathrm{e}^{(\alpha+\beta+1)G(u)}f(x,u)=0.
\end{equation}
Since $\Phi^{\prime}_{g}(t)=\mathrm{e}^{G(t)}$ and $\Phi^{\prime\prime}_{g}(t)=g(t)\Phi^{\prime}_{g}(t)$, it  is equivalent to the following
\begin{equation}\nonumber
   M(x,Du, D^2u)+g(u)N(x,Du, D^2u)+f(x,u)=0
\end{equation}
as desired.
\end{proof}     

\begin{proof}[Proof of Theorem~\ref{thm-VS}] Assume that $v:\Omega\to\mathbb{R}$ is a viscosity solution of \eqref{EWGV}. If $u=\Phi^{-1}_g(v)$, we have $u\in C(\Omega).$ Now, let  $\varphi\in C^{2}(\Omega)$ be such that $u-\varphi$ has a local maximum at $x_0$. In particular, we have an open set $\mathcal{O}\subset\Omega$ such that $u(x)\le u(x_0)+\varphi(x)-\varphi(x_0),\;\;\mbox{for all}\;\; x\in \mathcal{O}$. By setting $\overline{\varphi}=u(x_0)+\varphi-\varphi(x_0)$ in $\Omega$, we have 
$$
\overline{\varphi}\in C^{2}(\Omega)\;\; :\;\;u\le \overline{\varphi}\;\;\mbox{in}\;\;\mathcal{O}\;\;\mbox{and}\;\; u(x_0)=\overline{\varphi}(x_0).
$$
Since $\Phi_g$ is an increasing $C^2$ diffeomorphism, for $\psi=\Phi_g(\overline{\varphi})$ and $v=\Phi_g(u)$, we obtain 
$$
\psi\in C^{2}(\Omega)\;\; :\;\;v\le \psi\;\;\mbox{in}\;\;\mathcal{O}\;\;\mbox{and}\;\; v(x_0)=\psi(x_0).
$$
So, $v-\psi$ has a local maximum at $x_0$. Since $v$ is a viscosity subsolution of \eqref{EWGV} and $\psi=\Phi_g(\overline{\varphi})$, we have
\begin{equation}\nonumber
M(x_0,D\Phi_g(\overline{\varphi})(x_0),D^2\Phi_g(\overline{\varphi})(x_0))+h(x_0,v(x_0))\le 0.
\end{equation}
Thus, from Lemma~\ref{A-change} and using  $h(x,s)=\operatorname{e}^{(\alpha+\beta+1)G(\Phi^{-1}_g(s))}f(x,\Phi^{-1}_g(s))$, we can write 
\begin{equation}\nonumber
\begin{aligned}
&[\Phi^{\prime}_g(\overline{\varphi}(x_0))]^{\alpha+\beta+1}M(x_0,D\overline{\varphi}(x_0), D^2\overline{\varphi}(x_0))+\operatorname{e}^{(\alpha+\beta+1)G(u(x_0))}f(x_0,u(x_0))\\
&\quad +[\Phi^{\prime}_g(\overline{\varphi}(x_0))]^{\alpha+\beta}\Phi^{\prime\prime}_g(\overline{\varphi}(x_0))N(x_0,D\overline{\varphi}(x_0), D^2\overline{\varphi}(x_0))\le 0.  
\end{aligned}
\end{equation}
Recalling $\Phi^{\prime}_{g}(t)=\mathrm{e}^{G(t)}$ and $\Phi^{\prime\prime}_{g}(t)=g(t)\Phi^{\prime}_{g}(t)$ and noticing that $D\overline{\varphi}(x_0)=D\varphi(x_0)$, $D^2\overline{\varphi}(x_0)=D^2\varphi(x_0)$, and $\overline{\varphi}(x_0)=u(x_0)$,  we obtain 
\begin{equation}\nonumber
\begin{aligned}
M(x_0,D\varphi(x_0), D^2\varphi(x_0))+g(u(x_0))N(x_0, D\varphi(x_0), D^2\varphi(x_0))+f(x_0,u(x_0))\le 0.  
\end{aligned}
\end{equation}
It follows that $u=\Phi^{-1}_g(v)$ is a viscosity subsolution for \eqref{ECGV}. Analogously, we can see that $u=\Phi^{-1}_g(v)$ is a viscosity supersolution. 

Conversely, suppose that $u:\Omega\to\mathbb{R}$ is a viscosity solution of \eqref{ECGV} in $\Omega$ and let $v\in C(\Omega)$ be given by $v=\Phi_g(u)$. For any $\psi\in C^2(\Omega)$ such that $v-\psi$ has a local maximum at $x_0\in\Omega$, we have $v(x)\le v(x_0)+\psi(x)-\psi(x_0)$ for all $x\in\mathcal{O}$, for some open set $\mathcal{O}\subset\Omega$. Let us take $\overline{\psi}(x)=v(x_0)+\psi(x)-\psi(x_0)$ for all $x\in\Omega$. Then
$$
\overline{\psi}\in C^{2}(\Omega)\;\; :\;\;v\le \overline{\psi}\;\;\mbox{in}\;\;\mathcal{O}\;\;\mbox{and}\;\; v(x_0)=\overline{\psi}(x_0).
$$
Since $\Phi^{-1}_g$ is an increasing $C^2$ diffeomorphism, for $\varphi=\Phi^{-1}_g(\overline{\psi})$ and  noticing that $u=\Phi^{-1}_g(v)$, we obtain 
$$
\varphi\in C^{2}(\Omega)\;\; :\;\;u\le \varphi\;\;\mbox{in}\;\;\mathcal{O}\;\;\mbox{and}\;\; u(x_0)=\varphi(x_0).
$$
In particular, $u-\varphi$ has a local maximum at $x_0$. Since $u$ is a viscosity subsolution of \eqref{ECGV} and $u(x_0)=\varphi(x_0)$, we can write
\begin{equation}\label{voltaPI}
    \begin{aligned}
     M(x_0,D\varphi(x_0), D^2\varphi(x_0))+g(\varphi(x_0))N(x_0,D\varphi(x_0),D^2\varphi(x_0))+f(x_0,\varphi(x_0))\le 0.   
    \end{aligned}
\end{equation}
From Lemma~\ref{A-change} and using that $D\overline{\psi}(x_0)=D\psi(x_0)$, and $D^2\overline{\psi}(x_0)=D^2\psi(x_0)$, and $\overline{\psi}=\Phi_g(\varphi)$, we have
\begin{equation}\label{voltapII}  
\begin{aligned}
     M(x_0,D\psi(x_0), D^2\psi(x_0))&=M(x_0, D\Phi_g(\varphi)(x_0), D^2\Phi_g(\varphi)(x_0))\\
     &= [\Phi_{g}^{\prime}(\varphi(x_0))]^{\alpha+\beta+1}M(x_0,D\varphi(x_0), D^2\varphi(x_0))\\
     &+[\Phi_{g}^{\prime}(\varphi(x_0))]^{\alpha+\beta}\Phi_{g}^{\prime\prime}(\varphi(x_0))N(x_0,D\varphi(x_0), D^2\varphi(x_0))\\
     &= [\Phi_{g}^{\prime}(\varphi(x_0))]^{\alpha+\beta+1}\Big[M(x_0,D\varphi(x_0), D^2\varphi(x_0))\\
   &+ g(\varphi(x_0))N(x_0, D\varphi(x_0), D^2\varphi(x_0))\Big].
    \end{aligned}
\end{equation}
where we have used that $\Phi^{\prime\prime}_g(t)=g(t)\Phi^{\prime}_g(t)$. By combining \eqref{voltaPI} with \eqref{voltapII}, we get
\begin{equation}
    \begin{aligned}
         M(x_0,D\psi(x_0), D^2\psi(x_0))&\le -[\Phi_{g}^{\prime}(\varphi(x_0))]^{\alpha+\beta+1}f(x_0,\varphi(x_0))\\
         &=-\operatorname{e}^{(\alpha+\beta+1)G(\varphi(x_0))}f(x_0,\varphi(x_0))\\
         &=-h(x_0,v(x_0)).
    \end{aligned}
\end{equation}
where we have used that $\varphi(x_0)=\Phi^{-1}_{g}(\overline{\psi}(x_0))=\Phi^{-1}_{g}(v(x_0))$. Thus, $v=\Phi_g(u)$ is a viscosity subsolution of \eqref{EWGV}. A similar argument allows to conclude that $v=\Phi_g(u)$ is also a viscosity supersolution of \eqref{EWGV}. This conclude the proof. 
\end{proof}

\begin{flushleft}
{\bf Funding:}  
\begin{justify} This work was partially supported from CNPq through grants 309491/2021-5, 303443/2025-1.
\end{justify}
%  {\bf Competing interests:} The author declare that they have no
% competing interests. \\
%  {\bf Authors' contributions:}    All authors contributed to the study conception and design. All authors performed material preparation, data collection, and analysis. The authors read and approved the final manuscript.\\
% {\bf Availability of data and material:}  Not applicable.\\
% {\bf Ethical Approval:}  Not applicable.\\
% {\bf Consent to participate:}  All authors consent to participate in this work.\\
% {\bf Conflict of interest:} The authors declare no conflict of interest. \\
% {\bf Consent for publication:}  All authors consent for publication. \\
\end{flushleft}


\begin{thebibliography}{9}
\bibitem{MET2}
   B. Abdellaoui,  A. Dall'Aglio,  I. Peral,
   Some remarks on elliptic problems with critical growth in the gradient, 
   \emph{J. Differential Equations}, \textbf{222} (2006), 21-62.
 
 \bibitem{Adimurthi1}
    K. Adimurthi, C.P. Nguyen,  Quasilinear equations with natural growth in the gradients in spaces of Sobolev multipliers,
    \emph{Calc. Var.} \textbf{57} (2018)  


\bibitem{Aronsson}  G. Aronsson, On the partial differential equation $u^{2}_{x}u_{xx}+2u_{x}u_{y}u_{xy}+u^{2}_{y}u_{yy}=0$, \emph{Ark. Mat.}, 
\textbf{7} (1968) 395--425.

\bibitem{Aronsson2} G. Aronsson, M. G. Crandall, and P. Juutinen, A tour of the theory of absolutely minimizing functions, \textit{Bull. Amer. Math. Soc.}  \textbf{41} (2004),  439--505.

\bibitem{Arcoya}
    D. Arcoya, P. J. Martínez-Aparicio,
    Quasilinear equations with natural growth,
    \emph{Rev. Mat. Iberoam.}, \textbf{24} (2008), 597–616.

\bibitem{Bhatta} T. Bhattacharya, A. Mohammed,
Inhomogeneous Dirichlet problems involving the infinity-Laplacian, 
\emph{Adv. Differential Equations}, \textbf{17} (2012), 225-266

 \bibitem{Bhatta2} T. Bhattacharya, A. Mohammed, On solutions to Dirichlet problems involving the infinity-Laplacian, \emph{Adv. Calc. Var.}, \textbf{4} (2011), 445-487

\bibitem{Figa}
C. Bjorland, L. Caffarelli, A. Figalli, Nonlocal tug-of-war and the infinity fractional Laplacian,
\textit{Comm. Pure Appl. Math.} \textbf{65} (2012), 337-380.

\bibitem{Cabre}
    L.A. Caffarelli, X. Cabr\'{e}, Fully Nonlinear Elliptic Equation, \textit{American Mathematical Society}, 43, 1995.

\bibitem{Caselles}
   V. Caselles, J. . M. Morel, C. Sbert, An axiomatic approach to image interpolation,  \textit{IEEE Transactions on Image Processing},  \textbf{7}  (1998) 376--386 
   
\bibitem{Nosso}
    M. Cardoso, J. de Brito Sousa,   J.F. de Oliveira, 
    A gradient type term for the $k$-Hessian Equation,  \textit{J. Geom. Anal.} \textbf{34}, 19 (2024)

\bibitem{GuideV} M. G. Crandall, H. Ishii, and  P-L. Lions, User’s guide to viscosity solutions of second order partial differential equations, \textit{Bull. Amer. Math. Soc.} \textbf{27} (1992), 1-67
 
\bibitem{DeCoster}
  C. De Coster,  A. J. Fern\'{a}ndez, 
  Existence and multiplicity for elliptic $p$-Laplacian problems with critical growth in the gradient,\textit{ Calc. Var.},  \textbf{57} (2018), 57-89

\bibitem{Damiao}
D.J. Ara\'{u}jo, R. Leit\~{a}o, E. Teixeira,
Infinity Laplacian equation with strong absorptions,  \textit{J. Funct. Anal.} \textbf{270} (2016),  2249-2267


 \bibitem{ubilla}
    D.G. De Figueiredo, J.P. Gossez,  H.R. Quoirin,  P. Ubilla,
    Elliptic equations involving the $p$-Laplacian and a gradient term having natural growth,
    \emph{Rev. Mat. Iberoam.} (2019), 173-194
 
\bibitem{Monge}
J.F. de Oliveira, J.M. do \'{O}, P. Ubilla,
Boundary blow-up solutions for the Monge-Amp\`{e}re  equation with an invariant gradient type term, \textit{Applied Mathematics Letters}, 
\textbf{156}  (2024), 109141

\bibitem{Evans}
L. C. Evans, Estimates for smooth absolutely minimizing Lipschitz extensions, \textit{Electron. J.
Differential Equations}, (1993)

\bibitem{EvansSavin}
L.C. Evans, O. Savin,
$C^{1,\alpha}$ regularity for infinity harmonic functions in two dimensions, \textit{Calc. Var. Partial Differential Equations}, \textbf{32} (2008), 325–347.

\bibitem{EvansSmart}
 L.C. Evans, C.K. Smart, 
 Everywhere differentiability of infinity harmonic functions, \textit{Calc. Var. Partial Differential Equations}, \textbf{42} (2011), 289–299.

\bibitem{Garcia}
J. Garc\'{i}a-Mel\'{i}an, L. Iturriaga, H.R. Quoirin, A priori bounds and existence of solutions for slightly superlinear elliptic problems, \emph{Adv. Nonlinear Stud.} \textbf{15},  (2015), 923-938

\bibitem{Ishii} H. Ishii, P.L. Lions,
Viscosity solutions of fully nonlinear second-order elliptic partial differential equations, \textit{Journal of Differential Equations},
\textbf{83} (1990), 26--78.


\bibitem{IvoI} N. Ivochkina,  On differential equations of second order  with $d$-elliptic operators. \textit{Trans. Math. Inst. Steklova} \textbf{147}, (1980), 40--56 

\bibitem{IvoII} N. Ivochkina, N. Filimonenkova, On the backgrounds of the theory of $m$-Hessian equations. \textit{Commun. Pure Appl. Anal.} \textbf{12}, (2013), 1687--1703 

 \bibitem{JEANJEAN1}
    L. Jeanjean, H.R.  Quoirin,
    Multiple solutions for an indefinite elliptic problem with critical growth in the gradient,
    \emph{Proc. Amer. Math. Soc}. \textbf{144} (2016), 575-586.
    

\bibitem{JEANJEAN2}
    L. Jeanjean,  B. Sirakov,
    Existence and multiplicity for elliptic problems with quadratic growth in the gradient,
\emph{Comm. Partial Differential Equations}, \textbf{38} (2013) 244-264

 \bibitem{kazdan}
    J.L. Kazdan,  R.J. Kramer,
    Invariant criteria for existence of solutions to second--order quasilinear elliptic equations,
   \emph{Comm. Pure Appl. Math.}, \textbf{31} (1978), 619-645

\bibitem{GLu}
G. Lu, P. Wang, Inhomogeneous infinity Laplace equation, \textit{Adv. Math.} \textbf{217} (2008), 1838-1868.

    
\bibitem{Peres}
 Y. Peres, O. Schramm, S. Sheffield,  D. B. Wilson, Tug-of-war and the infinity Laplacian. \textit{J. Amer.
Math. Soc.}, \textbf{22} (2009), 167–210    
    
\bibitem{Savin}
O. Savin, 
$C^1$ regularity for infinity harmonic functions in two dimensions, \textit{Arch. Ration. Mech. Anal.} \textbf{176} (2005), 351–361.
   
\bibitem{Tso} K. Tso, Remarks on critical exponents for Hessian operators, \textit{Ann. Inst. H. Poincaré Anal. Non Linéaire.} \textbf{7} (1990), 113-122.

\bibitem{Yu} Y. Yu, A remark on $C^2$ infinity harmonic functions, \textit{Electronic Journal of Differential Equations},  \textbf{2006} (2006),  1–4.
 
\end{thebibliography}
\end{document}